\documentclass[reqno]{amsart}

\usepackage[english]{babel}

\usepackage{amsmath}
\usepackage{amssymb}
\usepackage{amsthm}

\usepackage{graphicx}
\usepackage[colorlinks=true, allcolors=blue]{hyperref}

\newtheorem{theorem}{Theorem}
\newtheorem{lemma}[theorem]{Lemma}
\newtheorem{proposition}[theorem]{Proposition}
\newtheorem*{problem}{Problem}

\theoremstyle{definition}

\begin{document}

\title[]{A lower bound for the\\ Balan--Jiang matrix problem}

\author[]{Afonso S.\ Bandeira}
\address{Department of Mathematics, ETH Z\"urich, Switzerland}

\author[]{Dustin G.\ Mixon}
\address{Department of Mathematics, The Ohio State University, Columbus, USA}

\author[]{Stefan Steinerberger}
\address{Department of Mathematics, University of Washington, Seattle, USA}

\begin{abstract} We prove the existence of a positive semidefinite matrix $A \in \mathbb{R}^{n \times n}$ such that any decomposition into rank-1 matrices has to have factors with a large $\ell^1-$norm, more precisely
$$ \sum_{k} x_k x_k^*=A  \quad \implies \quad \sum_k \|x_k\|^2_{1} \geq c \sqrt{n} \|A\|_{1},$$
where $c$ is independent of $n$. This provides a lower bound for the Balan--Jiang matrix problem. The construction is probabilistic.
\end{abstract}

\maketitle

\section{Introduction and Result}

Given a self-adjoint positive semidefinite matrix $A\in\mathbb{C}^{n\times n}$, consider the quantity
\[
\gamma_+(A)
:=\inf\left\{
\sum_{k=1}^p\|x_k\|_1^2:x_1,\ldots,x_p\in\mathbb{C}^n,\sum_{k=1}^p x_kx_k^*=A\right\}.
\]
At a recent AMS sectional meeting (Florida State, March 2024) and soon thereafter at an Oberwolfach meeting (April 2024), Radu Balan posed the following problem concerning the relationship between $\gamma_+$ and the entrywise $1$-norm $\|\cdot\|_1$.

\begin{problem}[Balan--Jiang Matrix Problem~\cite{BalanJ:24}]
For each $n\in\mathbb{N}$, determine 
\[
C_n
=\sup\left\{\frac{\gamma_+(A)}{\|A\|_1}:A\in\mathbb{C}^{n\times n},A\succeq 0,A\neq0\right\}.
\]
\end{problem}

This problem is motivated by an analogous problem posed by Hans Feichtinger at a 2004 Oberwolfach meeting (see \cite{heil}) in the setting of positive semidefinite trace class operators on $L^2(\mathbb{R}^d)$. One way of seeing that the constant $C_n$ exists is as follows: by taking the $\ell^2-$normalized eigenvectors of the matrix $A$, we have
$ A = \sum_{k=1}^{n} (\sqrt{\lambda_k} v_k)(\sqrt{\lambda_k} v_k)^*$
which is a decomposition with cost
$$ \sum_{k=1}^{n} \lambda_k \|v_k\|_{1}^2 \leq n  \sum_{k=1}^{n} \lambda_k \|v_k\|_{2}^2 =  n\sum_{k=1}^{n} \lambda_k = n \cdot \mbox{tr}(A) \leq n \|A\|_{1}.$$
This was noted by Balan--Jiang and shows $C_n \leq n$. Moreover, also observed by Balan-Jiang, when restricting to the eigendecomposition, these bounds can be seen to be optimal up to constants by taking a circulant matrix with small off-diagonal entries: circulants are diagonalized by the Fourier matrix for which $\|v_k\|_{1}^2 \sim n \|v_k\|_{2}^2$.
However, the eigendecomposition need not provide the optimal decomposition in this sense (and often does not). Our main contribution is a lower bound on $C_n$.
\begin{theorem}
\label{thm.main result} There exists a universal $c>0$ such that
$$ c \sqrt{n} \leq C_n \leq n.$$
\end{theorem}

\section{Proof}
\subsection{Preliminaries.}

Our proof makes use of a reduction due to Balan and Jiang~\cite{BalanJ:24}.
Given a self-adjoint matrix $T\in\mathbb{C}^{n\times n}$, consider the following quantities:
\begin{align*}
\pi_+(T)
&:=\sup\left\{\operatorname{Tr}(TA):A\in\mathbb{C}^{n\times n},A\succeq0,\|A\|_1=1\right\},\\
\rho_1(T)
&:=\sup\left\{\langle Tx,x\rangle:x\in\mathbb{C}^n,\|x\|_1\leq 1\right\}.
\end{align*}

\begin{proposition}[see~\cite{BalanJ:24}]
\label{prop.balan reduction} One has
$$\displaystyle C_n
=\sup\left\{
\frac{\pi_+(T)}{\rho_1(T)}:T\in\mathbb{C}^{n\times n},T^*=T,\rho_1(T)\neq0
\right\}.$$
\end{proposition}

If $T$ is negative definite, then $\rho_1(T)=0$.
Otherwise, any top eigenvector of $T$ with unit $1$-norm witnesses that $\rho_1(T)>0$, in which case any maximizer $x$ of $\langle Tx,x\rangle$ subject to $\|x\|_1\leq 1$ has unit $1$-norm.
Taking $A:=xx^*$ then gives a feasible point with the same value in the program that defines $\pi_+(T)$.
This implies $\pi_+(T)\geq\rho_1(T)$, and so $C_n\geq1$. Note that if $T$ vanishes on the diagonal, then the ratio is always $\leq 2$. One way of obtaining an improved lower bound on $C_n$ is to construct a matrix $T$ in such a way that 
\begin{itemize}
\item[(i)]
$T$ has nontrivial correlation with some positive semidefinite matrix so that
$$\pi_+(T)\geq\delta \qquad \mbox{for some uniform}~\delta > 0~\mbox{independent of}~n~\mbox{and}$$
\item[(ii)] $T$ has small correlation with every rank-$1$ matrix and $\rho_1(T)\to0$ as $n\to\infty$.
\end{itemize}

\subsection{The idea.}
The main idea is to consider a random matrix of the form
$$ T = - \frac{\sqrt{n}}{4} \cdot \mbox{Id}_{n \times n} + W,$$
where $W$ is a random symmetric matrix with Rademacher $\pm 1$ entries above the diagonal and 0 on the diagonal. The idea is as follows: if $x$ is spread out over many entries, then $\left\langle Wx, x \right\rangle$ is unlikely to be large: the randomness leads to too much cancellation. However, it is a priori conceivable that $x \in \mathbb{R}^n$ is carefully selected to exploit coherent substructures in $W$. In that case, $x$ is bound to be concentrated to have most of its support on a relatively small number of coordinates which then leads to large (negative) interactions with the identity matrix. The remainder of the argument makes this intuition precise.

\subsection{The first step.} We quickly prove that matrices $T$ of this type have $\pi_{+}(T)$ bounded away from 0 with high probability.
\begin{lemma} With high probability and $n$ sufficiently large,
$$ \pi_{+}(T) \geq \frac{1}{3}.$$
\end{lemma}
\begin{proof}
It suffices to find an explicit positive definite matrix $A \in \mathbb{R}^{n \times n}$ normalized to $\|A\|_1 = 1$. We consider matrices $A$ of the form
$$ A = a \cdot \mbox{Id}_{n \times n} + b \cdot W,$$
where $W$ is the same matrix as in the definition of $T$ and $a,b > 0$ as well as
$$ a n + b n(n-1) = 1$$
to ensure $\|A\|_1 = 1$.
 We have, using $\mbox{tr}(W)$ = 0 and $\mbox{tr}(W^2) = \|W\|_F^2$,
\begin{align*}
\mbox{tr}(TA) &= - a \frac{n^{3/2}}{4} + b \cdot \mbox{tr}(W^2) = - a \frac{n^{3/2}}{4} + b n(n-1) \\
&= 1 - a n   - a \frac{n^{3/2}}{4}.
\end{align*}
It remains to check whether we can choose $a>0$ to be sufficiently small for this quantity to remain bounded away from 0 while simultaneously ensuring that $A$ is positive definite. Here, we employ the Bai--Yin theorem \cite{BaiY:88} ensuring that
$$\lambda_{\max}(W)=(2+o(1))\sqrt{n}$$
happens with probability $1-o(1)$. This ensures that
\begin{align*}
 \lambda_{\min}(A) &\geq a -  (2+o(1))b\sqrt{n} = \alpha -  (2+o(1)) \frac{(1-a n)\sqrt{n}}{n (n-1)}.
 \end{align*}
Choosing $a = c n^{-3/2}$ for some $c>0$ then implies
\begin{align*}
\mbox{tr}(TA) &= 1 - \frac{c}{4} + o(1) \\
 \lambda_{\min}(A) &\geq (1+o(1))\frac{c-2}{n^{3/2}}.
\end{align*}
Any constant $2 < c < 4$ is an admissible choice for $n$ sufficiently large.
\end{proof}

\subsection{A Lemma for the second step.} The second step requires us to show that whenever $\|x\|_1 = 1$, then the quantity
$$\left\langle Tx, x \right\rangle  =   - \frac{\sqrt{n}}{4} \|x\|_2^2 +  \left\langle x, Wx \right\rangle$$
is small. Here, we will perform a decomposition of $x$ into coordinates with large entries and coordinates
with small entries: there can only be very few coordinates with large entries. The purpose of this
section is to show that when we restrict $W$ to few entries, its operator norm cannot be large.

\begin{lemma}
\label{lem.restricted spectral condition}
Let $W_S$ denote the principal submatrix of $W$ with index set $S$.
For every $\beta>0$ (independent of $n$), there exists $\alpha>0$ (independent of $n$) with
\[
\max_{\substack{S\subseteq[n]\\|S|\leq \alpha n}}\|W_S\|_{2\to2}
\leq \beta\sqrt{n}
\]
with probability at least $1-4e^{-\alpha \log(1/\alpha) n}$.
\end{lemma}

\begin{proof}
Given $\beta>0$, consider $\alpha>0$ to be specified later, and put $k:=\lfloor \alpha n\rfloor$.
Since $\|W_S\|_{2\to2}\leq \|W_T\|_{2\to2}$ whenever $S\subseteq T$, it suffices to show
\[
\max_{\substack{S\subseteq[n]\\|S|=k}}\|W_S\|_{2\to2}
\leq \beta\sqrt{n}.
\]
As a consequence of \cite[Theorem~4.4.5]{Vershynin:18} (see also \cite[Corollary~4.4.8]{Vershynin:18}), there exists a universal constant $C>0$ such that for any fixed $S\subseteq[n]$ with $|S|=k$, one has
$$\|W_S\|_{2\to2}\leq C(\sqrt{k}+t)$$
 with probability at least $1-4e^{-t^2}$.
A union bound then gives
\begin{align*}
\mathbb{P}\bigg\{
\max_{\substack{S\subseteq[n]\\|S|=k}}\|W_S\|_{2\to2}>C(\sqrt{k}+t)\bigg\}
&\leq \binom{n}{k}\cdot 4e^{-t^2}\\
&\leq 4\operatorname{exp}\Big(-t^2+k\log(n/k)\Big).
\end{align*}
Selecting $t=\sqrt{2\alpha n\log(1/\alpha)}$ then gives
\[
\mathbb{P}\bigg\{
\max_{\substack{S\subseteq[n]\\|S|=k}}\|W_S\|_{2\to2}>C(\sqrt{\alpha n}+\sqrt{2\alpha n\log(1/\alpha)})\bigg\}
\leq 4e^{-\alpha \log(1/\alpha) n}.
\]
As such, it suffices to select any $\alpha>0$ that satisfies 
$$C\sqrt{\alpha} (1+\sqrt{2\log(1/\alpha)})\leq\beta.$$
\end{proof}

\subsection{Proof of the Theorem: the second step}
We will use Lemma 2 only for the special case $\beta = 1/8$. In particular, we note the existence of a universal $\kappa > 0$ so that, with high probability,
$$\max_{\substack{S\subseteq[n]\\|S|\leq \kappa n}}\|W_S\|_{2\to2}
\leq \frac{1}{8} \sqrt{n}.
$$

We can now perform the second step of the construction.

\begin{lemma} With high probability,
$$ \rho_1(T) \leq\frac{34+o(1)}{\kappa^2} \frac{1}{\sqrt{n}}. $$
\end{lemma}
\begin{proof}
Let $x\in\mathbb{C}^n$ be the vector $\|x\|_1=1$ which maximizes $\left\langle Tx, x \right\rangle$ and let 
$$ S = \left\{1 \leq i \leq n: |x_i|\geq \frac{1}{\kappa n} \right\}$$
be the set of `large' entries and consider the decomposition $x=x_S+x_{S^c}$. We have 
\begin{align*}
\left\langle Tx, x \right\rangle &=  \left\langle T(x_S + x_{S^c}), (x_S + x_{S^c}) \right\rangle \\
&=  \left\langle Tx_S, x_S  \right\rangle + 2 \mbox{Re} \left\langle W x_S, x_{S^c} \right\rangle +   \left\langle Tx_{S^c}, x_{S^c} \right\rangle,
\end{align*}
where we used the fact that $x_S$ and $x_{S^c}$ are supported on disjoint coordinates to remove the contribution of the identity matrix to the off-diagonal. It remains to bound these three terms.\\

\textit{First term.} Note that, by $\ell^1-$normalization of $x$, we have $|S| \leq \kappa n$ and Lemma 2 applies. We deduce
\begin{align*}
\langle Tx_S,x_S\rangle &=-\frac{\sqrt{n}}{4}\|x_S\|_2^2+\langle Wx_S,x_S\rangle \\
&\leq-\frac{\sqrt{n}}{4}\|x_S\|_2^2+\|W_S\|_{2\to2}\|x_S\|_2^2 \leq-\frac{\sqrt{n}}{8}\|x_S\|_2^2.
\end{align*}

\textit{Second term.}  For the second term, we use again the Bai--Yin theorem to bound
\begin{align*}
2\operatorname{Re}\langle Wx_S,x_{S^c}\rangle \leq (4+o(1))\sqrt{n} \cdot \|x_S\|_2 \cdot \|x_{S^c}\|_2
\end{align*}
with high likelihood. Since all entries of $x_{S^c}$ are small, we have
$$  \|x_{S^c}\|_2 \leq \sqrt{ \frac{1}{\kappa^2 n^2} \cdot n } = \frac{1}{\kappa} \frac{1}{\sqrt{n}}$$
and we may bound the middle term by
$$ 2\operatorname{Re}\langle Wx_S,x_{S^c}\rangle \leq  \frac{(4+o(1))}{\kappa} \|x_S\|_2.$$

\textit{Third term.} The third term can be bounded by ignoring the negative contribution and using again the Bai--Yin result:
 \begin{align*}
 \left\langle Tx_{S^c}, x_{S^c} \right\rangle  &= - \frac{\sqrt{n}}{4} \|x_{S^c}\|_2^2 +  \left\langle W x_{S^c}, x_{S^c} \right\rangle \leq  \left\langle W x_{S^c}, x_{S^c} \right\rangle \\
 &\leq (2+o(1)) \sqrt{n} \|x_{S^c}\|_2^2 \leq \frac{2 + o(1)}{\kappa^2 \sqrt{n}}.
\end{align*}

\textit{Conclusion.} Finally, collecting all the terms, we have, with high probability,
\begin{align*}
\langle Tx,x\rangle
&\leq-\frac{\sqrt{n}}{8}\|x_S\|_2^2+\frac{4+o(1)}{\kappa}\|x_S\|_2+\frac{2+o(1)}{\kappa^2\sqrt{n}}.
\end{align*}
The polynomial 
$$ p(x) = -\frac{\sqrt{n}}{8}x^2+\frac{4+o(1)}{\kappa}x+\frac{2+o(1)}{\kappa^2\sqrt{n}}$$
assumes its unique global maximum in
$$ x = 4 \cdot \frac{4 + o(1)}{\kappa \sqrt{n}}$$
with maximal value
$$ p(x) \leq \frac{34 + o(1)}{\kappa^2} \frac{1}{\sqrt{n}}.$$
\end{proof}

Combining all the results, we deduce that, with high likelihood,
\[
C_n
\geq\frac{\pi_+(T)}{\rho_1(T)}
\geq\bigg(\frac{\kappa^2}{102}-o(1)\bigg) \sqrt{n}.
\]

\section{Remark}
Assume, for convenience, that $A \in \mathbb{R}^{n \times n}$ has real-valued entries. The eigendecomposition of $A$ shows that
$$ \gamma_{+}(A) \leq \sum_{k=1}^{n} \lambda_k \|v_k\|_{1}^2 \leq  \sum_{k=1}^{n} \lambda_k \|v_k\|_{2}^2 = n\sum_{k=1}^{n} \lambda_k = n \cdot \mbox{tr}(A).$$
We see that the second inequality is approximately sharp whenever $v_k$ has its $\ell^2-$mass evenly distributed across all entries. The purpose of this short section is to show that one can obtain the same type of bound by performing an iterative Matrix decomposition in the same style as in \cite{tropp}. More precisely, given an symmetric positive-definite matrix $A \in \mathbb{R}^{n \times n}$ with rows $a_1, \dots, a_n \in \mathbb{R}^n$, we may consider the new matrix
$$ A_2 = A - \frac{a_i a_i^{T}}{A_{ii}}.$$
The matrix $A_2$ remains symmetric and positive semi-definite. Moreover, we have
\begin{align*}
 \mbox{tr}(A_2) &=  \mbox{tr}(A) -  \mbox{tr}\left( \frac{a_i a_i^{T}}{A_{ii}}\right) \\
 &= \mbox{tr}(A) - \frac{1}{A_{ii}} \sum_{k=1}^{n} A_{ik}A_{ki} = \mbox{tr}(A) - \frac{\|a_i\|_2^2}{A_{ii}}.
\end{align*}
However, we also have
$$ A = A_2 + x x^* \quad \mbox{where} \quad x = \frac{1}{\sqrt{A_{ii}}} a_i \quad \mbox{and} \quad \|x\|_1^2 = \frac{\|a_i\|_1^2}{A_{ii}}.$$
We note that
$$ \|x\|_1^2 = \frac{\|a_i\|_1^2}{A_{ii}} = \frac{\|a_i\|_1^2}{\|a_i\|_2^2}  \frac{\|a_i\|_2^2}{A_{ii}} =  \frac{\|a_i\|_1^2}{\|a_i\|_2^2} \cdot \left( \mbox{tr}(A) - \mbox{tr}(A_2) \right)$$
and, in particular, by Cauchy--Schwarz,
$$ \|x\|_1^2 \leq  n \left( \mbox{tr}(A) - \mbox{tr}(A_2) \right).$$
Repeating this argument many times, we arrive at the following basic fact.
\begin{proposition}
Given a symmetric, positive-definite matrix $A \in \mathbb{R}^{n \times n}$, this decomposition, for any choice of indices $i$, results in a factorization
$$ \sum_{k=1}^n x_k x_k^*=A  \quad \mbox{with} \quad \sum_{k=1}^{n} \|x_k\|^2_{1} \leq n\cdot \emph{tr}(A).$$
\end{proposition}

To see that the inequality is sharp, consider the all-1's matrix $\mathbf{1} \mathbf{1}^T$.
However, we also see that the argument is typically lossy insofar as the ratio  $\|a_i\|_1^2/\|a_i\|_2^2$ is only close to $n$ whenever that row of the matrix is close to being constant. Numerical experiments suggest that this decomposition typically leads to much better upper bounds on $\gamma_{+}(A)$ than the eigendecomposition. It would be interesting if this could be made precise.\\

 \textbf{Acknowledgment.} The authors are grateful to Radu Balan for helpful discussions.

\end{document}